\documentclass[reqno,a4paper]{amsart}
\usepackage{amssymb}
\usepackage{amsmath}
\newcommand{\angles}[1]{\langle #1 \rangle}
\newcommand{\half}{\frac{1}{2}}

\newcommand{\norm}[1]{\left\Vert #1 \right\Vert}
\newcommand{\R}{\mathbb{R}}

\begin{document} 
\newtheorem{prop}{Proposition}[section]
\newtheorem{Def}{Definition}[section]
\newtheorem{theorem}{Theorem}[section]
\newtheorem{lemma}{Lemma}[section]
 \newtheorem{Cor}{Corollary}[section]

\title[LWP for Maxwell-Dirac]{\bf Local well-posedness for the Maxwell-Dirac system in temporal gauge}
\author[Hartmut Pecher]{
{\bf Hartmut Pecher}\\
Fakult\"at f\"ur  Mathematik und Naturwissenschaften\\
Bergische Universit\"at Wuppertal\\
Gau{\ss}str.  20\\
42119 Wuppertal\\
Germany\\
e-mail {\tt pecher@uni-wuppertal.de}}
\date{}

\begin{abstract}
We consider the low regularity well-posedness problem for the Maxwell-Dirac system in  n+1 dimensions in the cases $n=3$ and $n=2$ :
\begin{align*}
\partial^{\mu} F_{\mu \nu} & = - \langle \psi,\alpha_{\nu} \psi \rangle  \\
-i \alpha^{\mu} \partial_{\mu} \psi & = A_{\mu} \alpha^{\mu} \psi \, ,
\end{align*}
where $ F_{\mu \nu} = \partial^{\mu} A_{\nu} - \partial^{\nu} A_{\mu}$ , and
$\alpha^{\mu}$ are the Dirac matrices. We assume the temporal gauge $A_0=0$ and make use of the fact that some of the nonlinearities fulfill a null condition. Because we work in the temporal gauge we also apply  a method, which was used by Tao for the Yang-Mills system in this gauge. 
\end{abstract}
\maketitle
\renewcommand{\thefootnote}{\fnsymbol{footnote}}
\footnotetext{\hspace{-1.5em}{\it 2020 Mathematics Subject Classification:} 
35Q61, 35L70 \\
{\it Key words and phrases:} Maxwell-Dirac,
local well-posedness, temporal gauge}
\normalsize 
\setcounter{section}{0}

\section{Introduction and the main theorem}
The Maxwell–Dirac system describes the interaction of an electron
with its self-induced electromagnetic field. 
In Minkowski space $\R^{1+n}= \R_t \times \R^n_x$ , where we consider the cases $n=3$ and $n=2$ , it is given by
\begin{align}
\label{0.1}
\partial^{\mu} F_{\mu \nu} & = - \langle \psi,\alpha_{\nu} \psi \rangle  \\
\label{0.2}
-i \alpha^{\mu} \partial_{\mu} \psi & = A_{\mu} \alpha^{\mu} \psi \, ,
\end{align}
where 
$$ F_{\mu \nu} = \partial^{\mu} A_{\nu} - \partial^{\nu} A_{\mu} \, .$$ 
Greek indices run over $0 \le \mu \le n$ and Latin indices over $1 \le i \le n$ 
and the usual summation convention is used with the metric $diag(-1, 1, 1,1)$ and $diag(-1,1,1)$ on  $\R^{1+3}$ and $\R^{1+2}$, respectively.
In the case $n=3$ 
 the (4x4)  Dirac matrices $\alpha^{\mu}$ are 
$  \alpha^0 = \left( \begin{array}{cc}
I & 0  \\ 
0 & I  \end{array} \right)\, \,$ 
 , $\, \,  \alpha^j = \left( \begin{array}{cc}
0 & \sigma^j  \\
\sigma^j & 0  \end{array} \right) \, \, $, where $\sigma^j$ are the Pauli matrices $\, \, \sigma^1 = \left( \begin{array}{cc}
0 & 1  \\
1 & 0  \end{array} \right)$ ,
$ \sigma^2 = \left( \begin{array}{cc}
0 & -i  \\
i & 0  \end{array} \right)$ ,
$ \sigma^3 = \left( \begin{array}{cc}
1 & 0  \\
0 & -1  \end{array} \right)$ . \\
In the case $n=2$ we define $\alpha^0 = I$ and $\alpha^i = \sigma^i$ for $i=1,2$ .

$\alpha^{\mu}$ are hermitian matrices with $(\alpha^{\mu})^2 = I$ , $\alpha^j \alpha^k + \alpha^k \alpha^j = 0$ for $j \neq k$ .
The functions $A_{\mu} : \R^{1+n} \to \R$ are the potentials,  $\psi: \R^{1+n} \to \mathbb{C}^N$  is the spinor and $\langle \cdot,\cdot \rangle $ is the $\mathbb{C}^N$ inner product, where $N=4$ for $n=3$ and $N=2$ for $n=2$ .

We omit a mass term $m \psi$ in the Dirac equation just for convenience.

Our aim is to obtain a well-posedness result for the Cauchy problem under minimal regularity assumptions on the data.

 As is well-known
the system (\ref{0.1}),(\ref{0.2}) is invariant under the gauge transformations
$$A_{\mu} \to A'_{\mu} = A_{\mu} + \partial_{\mu} \chi \, , \, \psi \to \psi' = e^{i \chi} \psi \, , \, D_{\mu} \to D'_{\mu} = \partial_{\mu} - i A'_{\mu} \, .$$
This allows to impose an additional gauge condition.

From now on we assume the temporal gauge $A_0 =0$ .

In this case equation (\ref{0.1}) reduces to
\begin{align}
\label{2'}
\partial_t\,  div \, A & = - \langle \psi,\psi \rangle \, , \\
\label{1'}
\square A_j - \partial^j (div\, A) & = - \langle \psi,\alpha_j \psi \rangle \, .
\end{align}

We apply the well-known Hodge decomposition $ A = A^{df} + A^{cf}$ , where the divergence-free part is given by $PA:= A^{df}: = |\nabla|^{-2} \nabla \times (\nabla \times A)$ for $n=3$ and  $PA:= A^{df}: = (R_2(R_1 A_2-R_2 A_1),-R_1(R_1 A_2 - R_2 A_1))$ for $n=2$ , and the curl-free part by $(1-P)A := A^{cf} :=  -|\nabla|^{-2} \nabla(div \, A)$ \, .

If we apply the projection $P$ of $A$ onto the divergence-free part $A^{df}$ we obtain from (\ref{1'}) :
\begin{equation}
\label{1''}
\square A^{df} = - P \left(\begin{array}{c} \langle \psi,\alpha_1\psi \rangle \\ \langle \psi,\alpha_2 \psi \rangle\\ \langle \psi,\alpha_3\psi \rangle \end{array} \right) 
\end{equation}
for $n=3$ and
\begin{equation}
\label{1'''}
\square A^{df} = - P \left(\begin{array}{c} \langle \psi,\alpha_1\psi \rangle \\ \langle \psi,\alpha_2 \psi \rangle \end{array} \right) 
\end{equation}
for $n=2$ . \\
By the definition of $A^{cf}$ the equation (\ref{2'}) can be rewritten as
\begin{equation}
\label{2*}
\partial_t A^{cf} = -|\nabla|^{-2} \nabla \langle \psi, \psi \rangle \, . \end{equation} 
We reformulate (\ref{1''}) as a first order system
\begin{equation}
\label{1}
(-i \partial_t \pm \langle \nabla \rangle ) A^{df}_{\pm} = \mp 2^{-1} \langle \nabla \rangle^{-1} P \left(\begin{array}{c} \langle \psi,\alpha_1\psi \rangle \\ \langle \psi,\alpha_2 \psi \rangle\\ \langle \psi,\alpha_3\psi \rangle \end{array} \right) - A_j \, ,
\end{equation}
where we define
$A^{df}_{\pm} = \half(A^{df}\pm (i \langle \nabla \rangle)^{-1}(\partial_t A^{df})) \, , $
so that $A^{df} = A_{+}^{df} + A_{-}^{df}$ , $\partial_t A^{df} = i \langle \nabla \rangle (A_{+}^{df} - A_{-}^{df})$ . Similarly (\ref{1'''}) is rewritten.

Following \cite{AFS1} and \cite{HO} in order to  rewrite the Dirac equation we define the projections
$$\Pi(\xi) := \half(I + \frac{\xi_j \alpha^j}{|\xi|})$$  and $\Pi_{\pm}(\xi):= \Pi(\pm \xi)$ , so that $\Pi_{\pm}(\xi)^2 = \Pi_{\pm}(\xi)$ , $\Pi_+(\xi) \Pi_-(\xi) =0 $ , $\Pi_+ (\xi) + \Pi_-(\xi) = I$ , $\Pi_{\pm}(\xi) = \Pi_{\mp}(-\xi) $ .
We obtain 
\begin{equation}
\label{2.6}
\alpha^j \Pi(\xi) = \Pi(- \xi) \alpha^j + \frac{\xi_j}{|\xi|} I_{4x4} \, .
\end{equation}
Using the notation $\Pi_{\pm} = \Pi_{\pm}(\frac{\nabla}{i})$ we obtain
\begin{equation}
\label{2.8}
 -i\alpha^j \partial_j = |\nabla|\Pi_+ - |\nabla|\Pi_- \, , 
\end{equation}
where $|\nabla|$ has symbol $|\xi| $ . Moreover defining the modified Riesz transform by $R^j_{\pm} = \mp(\frac{\partial_j}{i|\nabla|}) $ with symbol $ \mp \frac{\xi_j}{|\xi|}$ the identity (\ref{2.6}) implies
\begin{equation}
\label{2.7}
\alpha^j \Pi_{\pm} = (\alpha^j \Pi_{\pm})\Pi_{\pm} = \Pi_{\mp} \alpha^j \Pi_{\pm}- R^j_{\pm} \Pi_{\pm}  \, .
\end{equation}
If we define $\psi_{\pm} = \Pi_{\pm} \psi$ we obtain by applying the projection $\Pi_{\pm}$ and (\ref{2.8}) the Dirac type equation in the form
\begin{equation}
\label{3}
(i \partial_t \pm |\nabla|)\psi_{\pm} = \Pi_{\pm}(A_j^{df} \alpha^j   \psi + A^{cf}_j \alpha^j \psi)  \, .
\end{equation}

We want to solve the Cauchy problem for the system (\ref{1}),(\ref{2*}) and (\ref{3}) .    This system is certainly equivalent to the original system (\ref{0.1}),(\ref{0.2}) in temporal gauge $A_0=0$ .

We consider initial data for $A_j$ and $\psi$ at $t=0$:
\begin{equation}\label{Data}
A_j(0) = a_j , \quad (\partial_t A_j)(0) =  b_j,
         \quad \psi(0) = \psi_0.
 \end{equation}
 Equation (\ref{2'}) requires  the compatability condition
\begin{equation}
\label{CC}
\partial^j b_j = -|\psi_0|^2 \, ,
\end{equation}
which we assume from now on.

For our further considerations it is also important that the gauge invariance allows to assume 
\begin{equation}
\label{1.14}
A^{cf}(0) = A^{cf}(0,x) = 0 \, .
\end{equation}
One only has to choose 
\begin{equation}
\label{GT}
 \chi(x) = |\nabla|^{-2} div\, A(0,x) \, .
 \end{equation}
This implies 
$ A_0' = A_0 =0$ and $$A'^{cf}(0)= A^{cf}(0) + \nabla \chi= -|\nabla|^{-2} \nabla\, div\, A(0) +  |\nabla|^{-2} \nabla \,div\, A(0) = 0 \, . $$
These data transform as follows:
\begin{align}
\nonumber
 a_{j,\pm} = A_{j,\pm}(0) = \half(a_j \pm (i \langle \nabla \rangle )^{-1} b_j)\, , \quad & \quad
  \psi_{\pm}^0 =\psi_{\pm}(0)  = \Pi_{\pm} \psi_0 \, . 
\end{align}

Our main theorem reads as follows:
\begin{theorem}
\label{Theorem1.1}
Assume that $s$ , $r$ and $l$ satisfy the following conditions in the case $n=3$:
\begin{align*}
&s > \frac{1}{4} \, ,  \quad r >\frac{5}{8} \, , \quad l > 1 \\
& r \ge s \ge r-1 \, , \, l \ge s \ge l-1 \, , \, \\
&2s-r > -\frac{1}{8}\, ,  \,  3s-2r > -1 \, , \, 2r-s > 1  \, , \, s-l > -\frac{3}{4} \, . 
\end{align*}
In the case $n=2$ we assume 
\begin{align*}
&s > 0 \, ,  \quad r >\frac{1}{4} \, , \quad l > \half \\
& r \ge s \ge r-1 \, , \, l \ge s \ge l-1 \, , \, \\
&2s-r > -\frac{3}{4}\, ,  \,  3s-2r > -\frac{3}{2} \, , \, 4s-r > -\half \, , \,  2r-s > \half \, , \, s-l \ge -\frac{1}{2} \, . 
\end{align*}
Given initial data $a_j = a_j^{df} + a_j^{cf}$ with $a_j^{df} \in H^r$ , $a_j^{cf} = 0$ , and $b_j \in H^{r-1}$ , $\psi_0 \in H^s$, which fulfill (\ref{CC}) ,  there exists a time $T > 0$ , depending on the norms of the data, such that the Cauchy problem  (\ref{Data}) for  the Maxwell-Dirac system (\ref{0.1}),(\ref{0.2}) in temporal gauge $A_0=0$ has a unique solution $$A_j^{df} \in  X^{r,b}_{+}[0,T]+X^{r,b}_{-}[0,T] \, , \, \psi \in X^{s,\frac{1}{2}+}_+[0,T] + X^{s,\frac{1}{2}+}_-[0,T]    \, , $$
and
$$ A_j^{cf}  \in X^{l,\half+}_{\tau=0}[0,T] \,\, \text{for}\,\, n=3 \, , \,  |\nabla|^{\epsilon} A_j^{cf}  \in X^{l-\epsilon,\half+}_{\tau=0}[0,T] \,\, \text{for} \,\, n=2 \, ,$$  where $\epsilon > 0$ is sufficiently small, $b=\frac{7}{8}+$ for $n=3$ and $b=\half+$ for $n=2$ (these spaces are defined in Def. \ref{Def.1.2}). It has the regularity
$$ A^{df} \in C^0([0,T],H^r) \cap C^1([0,T],H^{r-1}) \, , \, \,  
 \psi \in C^0([0,T],H^s) \, , $$
$$ A^{cf} \in C^0([0,T],H^l) \, \, \text{for} \,\, n=3 \, , \, |\nabla|^{\epsilon} A^{cf} \in C^0([0,T],H^{l-\epsilon}) \,\, \text{for} \,\, n=2 \, . $$
The solution depends continously on the data and higher regularity persists.
\end{theorem}
\noindent{\bf Remarks:} 
1. In the case $n=2$  we replace $A^{cf}_j$ by $|\nabla|^{\epsilon} A^{cf}_j$ in order to avoid the operator $|\nabla|^{-1}$ in equation (\ref{2*}) and replace it by $|\nabla|^{-1+\epsilon}$ , which is less singular   . \\
2. For $n=3$ the choice $s=\frac{1}{4}+\delta$ , $r=\frac{5}{8} + \delta$ , $l = 1+\delta$ for arbitrary $\delta > 0$ obviously fulfills our assumptions , and for $n=2$ the choice $s=\delta$ , $r=\frac{1}{4}+\delta$ , $l = \half + \delta$ .\\[0.3em]

We prove local well-posedness by iterating in  $X^{s,b}$-spaces adapted to the  operators $i\partial_t \pm \langle \nabla \rangle $ and $\partial_t$ . 
\begin{Def}
\label{Def.1.2} 
$X^{s,b}_{\pm}$ is the completion of $\mathcal S(\R^{1+n})$ with respect to the norm
$$
  \norm{u}_{X^{s,b}_\pm} = \| \angles{\xi}^s \langle -\tau \pm |\xi| \rangle^b \widetilde u(\tau,\xi) \|_{L^2_{\tau,\xi}},
$$
where $\widetilde u(\tau,\xi) = \mathcal F_{t,x} u(\tau,\xi)$ is the space-time Fourier transform of $u(t,x)$.

 Let $X^{s,b}_\pm [0,T]$ denote the restriction space to the interval $[0,T]$ for $T>0$.
In addition to $X^{s,b}_\pm$, we shall also need the wave-Sobolev spaces $X^{s,b}_{|\tau|=|\xi|}$, defined to be the completion of $\mathcal S(\R^{1+n})$ with respect to the norm
$$
  \norm{u}_{X^{s,b}_{|\tau|=|\xi|}} = \norm{\angles{\xi}^s \langle |\tau|-  |\xi| \rangle^b \widetilde u(\tau,\xi)}_{L^2_{\tau,\xi}}.
$$ 
and $X^{s,b}_{\tau = 0}$ is the completion of  $\mathcal S(\R^{1+3})$ with respect to the norm
$$\|u\|_{X^{s,b}_{\tau=0}} = \| \langle \xi \rangle^s \langle \tau \rangle^b  \widetilde u(\tau,\xi)\|_{L^2_{\tau,\xi}} $$
and $X^{s,b}_{\tau=0}[0,T]$ its restriction to $[0,T]$ .
\end{Def}
We recall the fact that
 $$
 X^{s,b}_\pm [0,T] \,,\, X^{s,b}_{\tau=0}[0,T]  \hookrightarrow C^0([-T,T];H^s) \quad \text{for} \ b > \half.
$$

Let us make some historical remarks. As is well-known we may impose a gauge condition. We exlusively study the temporal gauge $A_0 =0 $ . Other convenient gauges are the Coulomb gauge $\partial^j A_j=0$ and the Lorenz gauge $\partial^{\alpha}A_{\alpha} =0$. It is well-known that for the low regularity well-posedness problem for the Maxwell-Klein-Gordon equation (and the Yang-Mills equation) a null structure for the nonlinear terms plays a crucial role. This was first detected by Klainerman and Machedon \cite{KM}, who proved global well-posedness in the case of three  space dimensions in Coulomb gauge for finite energy data. These null condition also plays a decisive role for many other systems of mathematical physics. In this paper we are interested in the Maxwell-Dirac system in three and two space dimensions. In three space dimensions Bournaveas \cite{B} considered this system in Coulomb gauge and proved local well-posedness for $A_j(0) \in H^2$, $(\partial_t A_j)(0) \in H^1$ and $\psi(0) \in H^1$. Moreover in Lorenz gauge he only assumed  $A_{\mu}(0) \in H^{1+\epsilon}$, $(\partial_t A_{\mu})(0) \in H^{\epsilon}$ and $\psi(0) \in H^{\half+\epsilon}$ with $\epsilon > 0$ . This result was improved by Masmoudi-Nakanishi \cite{MN1} in Coulomb gauge who proved local well-posedness for finite energy data in the classical solution space $A \in C^0([0,T],\dot{H}^1) \cap C^1([0,T],L^2)$, $ \psi \in C^0([0,T],L^2)$ . These authors also studied the nonrelativistic limit of this system as well as the Maxwell-Klein-Gordon system as $c \to \infty$ (cf. \cite{MN}).
An almost optimal local well-posedness result in Lorenz gauge , namely for data $\psi(0) \in H^s$  and $F_{\mu \nu}(0) \in H^{s-\half}$ for $s>0$ was obtained  by d'Ancona, Foschi and Selberg \cite{AFS1} who detected a new null structure of the system as a whole.  All these results are given in 3+1 dimensions.

In 2+1 dimensions the fundamental global well-posedness result for data $\psi(0) \in L^2$ and (essentially) $F_{\mu \nu}(0) \in H^{\half}$ was proven by d'Ancona and Selberg \cite{AS}.

An ill-posedness result in dimensions $n \le 3$ for the case $s < 0$ was recently obtained by Selberg and Tesfahun \cite{ST1}, which means that the results in \cite{AS} and \cite{AFS1} are optimal and almost optimal in 2+1 and 3+1 dimensions, respectively.

Important for the present paper are the methods developed by Tao \cite{T} for a small data local well-posedness result for the Yang-Mills equations. We also rely on the methods used by Huh-Oh \cite{HO} for the Chern-Simons-Dirac equation.

In the present paper we consider the coupled Maxwell-Dirac equation in temporal gauge for the space dimension $n=3$ . To the best of our knowledge  we obtain the first local well-posedness result for the temporal gauge. Our aim is to minimize the regularity of the Cauchy data. If we assume $A^{df}(0) \in H^r$ , $A^{cf}(0)=0$ (this condition which may be assumed by gauge invariance) , and $\psi(0) \in H^s$ , we obtain a solution of the Cauchy problem by a Picard iteration with $ A^{df} \in C^0([0,T],H^r) \cap C^1([0,T],H^{r-1})  \, , \,  \psi \in C^0([0,T],H^s) \, , $ and $ A^{cf} \in C^0([0,T],H^l) $ for $n=3$ , $ |\nabla|^{\epsilon} A^{cf} \in C^0([0,T],H^{l-\epsilon}) $ for $n=2$ , where $\epsilon >0$ is suffciently small,
provided $s > \frac{1}{4}$ , $r > \frac{5}{8}$ , $l > 1$ 
for $n=3$ , and $s>0$ , $r > \frac{1}{4}$ , $l> \half$ for $n=2$ . Here $A^{df}$ and $A^{cf}$ denote the divergence-free part and the curl-free part respectively.
We remark, that this is almost optimal with respect to the data for the spinor for $n=2$ , if we compare it with the ill-posedness \cite{ST1} mentioned before. Uniqueness holds in spaces of $X^{s,b}$-type, which are the spaces where the fixed point argument works. The null conditions detected by Klainerman, Machedon, Selberg, Huh, Oh and others are of course fundamental for the necessary bilinear estimates. As mentioned before Tao`s methods for the Yang-Mills equation are fundamental as well as the convenient  atlas of bilinear estimates in wave-Sobolev spaces by \cite{AFS0}, \cite{AFS} and \cite{ST}.

\section{Reduction to multilinear estimates}
  
It is well known that the linear initial value problem
 $$
  (i\partial_t \pm \langle \nabla \rangle ) u = G  \in X^{s,b-1+\delta}_\pm [0,T], \qquad u(0) = u_0\in H^s,
$$
for any $s\in \R, \ b > \frac12$,  $\ 0<\delta \ll 1$,  
has a unique solution
satisfying
 \begin{equation}\label{LinearEst}
\norm{u}_{X^{s,b}_\pm [0,T]} \le C \left( \norm{u_0}_{H^s} + T^{\delta} \norm{G}_{X^{s,b-1+\delta}_\pm [0,T]} \right)
\end{equation}  
for $0<T<1$.

A similar result holds for the equation $\partial_t u = G$ and the space $X^{s,b}_{\tau=0}$ .\\[0.5em]

This implies that in order to prove Theorem \ref{Theorem1.1} by an iteration argument it suffices to prove the following bilinear estimates:

\begin{prop}
\label{Prop.1}
Assume $n=3$ . Let the assumptions on  $s$,$r$ and $l$ in Theorem \ref{Theorem1.1} be fulfilled. Let $ b = \frac{7}{8}+$ . Then the following estimates apply:
\begin{align}
\label{4.1}
\| P \langle \psi_{\pm_1} , \alpha^j \psi_{\pm_2} \rangle \|_{X^{r-1,b-1+}_{\pm}} & \lesssim \|\psi_{\pm_1}\|_{X^{s,\half+}_{\pm_1}} \|\psi_{\pm_2}\|_{X^{s,\half+}_{\pm_2}} \, , \\
\label{4.2}
\| |\nabla|^{-1} \langle \psi_{\pm_1} , \psi_{\pm_2} \rangle \|_{X^{l,-\half++}_{\tau=0}} & \lesssim \|\psi_{\pm_1}\|_{X^{s,\half+}_{\pm_1}} \|\psi_{\pm_2}\|_{X^{s,\half+}_{\pm_2}} \, , \\
\label{4.3}
\| \Pi_{\pm} (A^{cf}_j \alpha^j \psi) \|_{X^{s,-\half++}_{|\tau|=|\xi|}} & \lesssim \|A^{cf}\|_{X^{l,\half+}_{\tau=0}} \|\psi\|_{X^{s,\half+}_{|\tau|=|\xi|}} \, , \\
\label{4.4}
\| \Pi_{\pm} (A^{df}_{j,\pm_1} \alpha^j \psi_{\pm_2}) \|_{X^{s,-\half++}_{\pm}} & \lesssim (\|A^{df}_{j\pm_1}\|_{X^{r,b}_{\pm_1}} + \|A^{cf}_j\|_{X^{l,\half+}_{\tau =0}}) \|\psi_{\pm_2}\|_{X^{s,\half+}_{\pm_2}}\, , 
\end{align}
where $\pm$ , $\pm_1$ and $\pm_2$ denote independent signs.
\end{prop}

The following product estimates for wave-Sobolev spaces were proven in \cite{AFS}. 

\begin{prop}
\label{Prop.3.3}
For $s_0,s_1,s_2,b_0,b_1,b_2 \in {\mathbb R}$ and $u,v \in   {\mathcal S} ({\mathbb R}^{3+1})$ the estimate
$$\|uv\|_{H^{-s_0,-b_0}} \lesssim \|u\|_{H^{s_1,b_1}} \|v\|_{H^{s_2,b_2}} $$ 
holds, provided the following conditions are satisfied:
\begin{align*}
\nonumber
& b_0 + b_1 + b_2 > \frac{1}{2} \, ,
& b_0 + b_1 \ge 0 \, ,\quad \qquad  
& b_0 + b_2 \ge 0 \, ,
& b_1 + b_2 \ge 0
\end{align*}
\begin{align*}
\nonumber
&s_0+s_1+s_2 > 2 -(b_0+b_1+b_2) \\
\nonumber
&s_0+s_1+s_2 > \frac{3}{2} -\min(b_0+b_1,b_0+b_2,b_1+b_2) \\
\nonumber
&s_0+s_1+s_2 > 1 - \min(b_0,b_1,b_2) \\
\nonumber
&s_0+s_1+s_2 > 1 \\
 &(s_0 + b_0) +2s_1 + 2s_2 > \frac{3}{2} \\
\nonumber
&2s_0+(s_1+b_1)+2s_2 > \frac{3}{2} \\
\nonumber
&2s_0+2s_1+(s_2+b_2) > \frac{3}{2}
\end{align*}
\begin{align*}
\nonumber
&s_1 + s_2 \ge \max(0,-b_0) \, ,\quad
\nonumber
s_0 + s_2 \ge \max(0,-b_1) \, ,\quad
\nonumber
s_0 + s_1 \ge \max(0,-b_2)   \, .
\end{align*}
\end{prop}

A consequence of this result is the following analogue for null forms.

\begin{prop}[Null form estimates, \cite{ST} ]  
\label{Prop.1.2}
Let $\sigma_0,\sigma_1,\sigma_2,\beta_0,\beta_1,\beta_2 \in \R$ and $u,v \in   {\mathcal S} ({\mathbb R}^{3+1})$ . Assume that
\begin{equation*}
\left\{
\begin{aligned} 
 & 0 \le \beta_0 < \frac12 < \beta_1,\beta_2 < 1,
  \\
 & \sum \sigma_i + \beta_0 > \frac32 - (\beta_0 + \sigma_1 + \sigma_2),
  \\
 & \sum \sigma_i > \frac32 - (\sigma_0 + \beta_1 + \sigma_2),
  \\
&  \sum \sigma_i > \frac32 - (\sigma_0 + \sigma_1 + \beta_2),
  \\
  &\sum \sigma_i + \beta_0 \ge 1,
  \\
 & \min(\sigma_0 + \sigma_1, \sigma_0 + \sigma_2, \beta_0 + \sigma_1 + \sigma_2) \ge 0,
\end{aligned}
\right.
\end{equation*} 
and that the last two inequalities are not both equalities. Let
\begin{align}
\nonumber
&{\mathcal F}(B_{\pm_1,\pm_2} (\psi_{1_{\pm_1}}, \psi_{2_{\pm_2}}))(\tau_0,\xi_0) \\
\label{2}
& := \int_{\tau_1+\tau_2= \tau_0\, \xi_1+\xi_2=\xi_0} |\angle(\pm_1 \xi_1,\pm_2 \xi_2)|  \widehat{\psi_{1_{\pm_1}}}(\tau_1,\xi_1)  \widehat{\psi_{2_{\pm_2}}}(\tau_2,\xi_2) d\tau_1 d\xi_1 \, , 
\end{align}
where $\angle(\xi_1,\xi_2)$ denotes the angle between $\xi_1$ and $\xi_2$ .
Then we have the null form estimate
$$
  \norm{B_{(\pm_1 \xi_1,\pm_2 \xi_2)}(u,v)}_{H^{-\sigma_0,-\beta_0}}
  \lesssim
  \norm{u}_{X^{\sigma_1,\beta_1}_{\pm_1}} \norm{v}_{X^{\sigma_2,\beta_2}_{\pm 2}}\, .
$$
\end{prop}

We also need the following variant of this result.
\begin{lemma} 
\label{Lemma}
Let  $s_0+s_1\ge 0$ , $s_0+s_2 \ge 0 $ , $s_1+s_2+\half \ge 0$ , $0 \le b_2 \le \half$  and $u,v \in   {\mathcal S} ({\mathbb R}^{3+1})$ . Assume that
$$s_0+s_1+s_2 +b_2 > 1 \, , $$
$$ (s_0+s_1+s_2+b_2)+(s_0+s_1) >\frac{3}{2} \, , $$
$$(s_0+s_1+s_2+\half) +(s_0+s_2) > \frac{3}{2}$$
Then we have the following estimate
$$
  \norm{B_{(\pm_1 \xi_1,\pm_2 \xi_2)}(u,v)}_{H^{-s_0,-\half}}
  \lesssim
  \norm{u}_{X^{s_1,\half}_{\pm_1}} \norm{v}_{X^{s_2,b_2}_{\pm 2}}\, .
$$
\end{lemma}
\begin{proof}
Combine Prop. \ref{Prop.3.3} and the following estimate for the angle:
\begin{equation}
\label {angle}
 \angle(\pm_1 \xi_1,\pm_2 \xi_2)
 \lesssim \left(\frac{\langle - \tau_1 \pm_1 |\xi_1| \rangle }{\min(\langle \xi_1 \rangle,\langle \xi_2 \rangle)} \right)^{\half} + \left(\frac{\langle - \tau_2 \pm_2 |\xi_2| \rangle }{\min(\langle \xi_1 \rangle,\langle \xi_2 \rangle)} \right)^{b_2} + \left(\frac{\langle |\tau_3| -|\xi_3|| \rangle }{\min(\langle \xi_1 \rangle,\langle \xi_2 \rangle)} \right)^{\half}  ,
\end{equation}
where $\xi_j \in \R^3$ , $\tau_j \in \R$ with $\xi_1+\xi_2+\xi_3=0$ , $\tau_1+\tau_2+\tau_3=0$ . For a proof we refer to Selberg \cite{S}, Lemma 2.1. This implies that it is sufficient to prove the following estimates:
\begin{align*}
\|uv\|_{H^{-s_0,0}} & \lesssim \|u\|_{H^{s_1+\half,\half}} \|v\|_{H^{s_2,b_2}} \, ,\\
\|uv\|_{H^{-s_0,0}} & \lesssim \|u\|_{H^{s_1,\half}} \|v\|_{H^{s_2+\half,b_2}} \, ,\\
\|uv\|_{H^{-s_0,-\half}} & \lesssim \|u\|_{H^{s_1,0}} \|v\|_{H^{s_2+\half,b_2}} \, ,\\
\|uv\|_{H^{-s_0,-\half}} & \lesssim \|u\|_{H^{s_1+\half,0}} \|v\|_{H^{s_2,b_2}} \, , \\
\|uv\|_{H^{-s_0,-\half}} & \lesssim \|u\|_{H^{s_1+b_2,\half}} \|v\|_{H^{s_2,0}} \, ,\\
\|uv\|_{H^{-s_0,-\half}} & \lesssim \|u\|_{H^{s_1,\half}} \|v\|_{H^{s_2+b_2,0}} \, .
\end{align*}
Using Prop. \ref{Prop.3.3} these estimates are fulfilled , if $s_0+s_1+s_2+b_2 > 1$ and moreover $(s_0+s_1+\half+s_2)+(s_0+s_1+b_2) > \frac{3}{2}$ , $(s_0+s_1+\half+s_2)+(s_1+s_2+\half) > \frac{3}{2}$ ,  $(s_0+s_1+\half+s_2)+(s_0+s_2) > \frac{3}{2}$ for the first four estimates. The last two estimates  require $(s_0+s_1+s_2+b_2) +(s_0+s_1) > \frac{3}{2}$ . These conditions are fulfilled under our assumptions.
\end{proof}

The following estimate was used by \cite{T} for the Yang-Mills equation in temporal gauge.
\begin{prop}
\label{Prop.2.4}
The following estimates hold for $n=3$ :
\begin{align}
\label{3.1}
\|u\|_{L^p_x L^2_t} & \lesssim \|u\|_{X^{1-\frac{3}{p},\frac{1}{2}+}_{|\tau|=|\xi|}} \quad \mbox{for} \,\, 4 \le p < \infty \, ,\\
\label{3.2}
\|u\|_{L^{\infty}_x L^2_t} &\lesssim \|u\|_{X^{1+,\frac{1}{2}+}_{|\tau|=|\xi|}} \, , \\
\label{3.3} 
\|u\|_{L^p_x L^{2+}_t} & \lesssim \|u\|_{X^{1-\frac{3}{p}+,\frac{1}{2}+}_{|\tau|=|\xi|}}  \quad \mbox{for} \,\, 4 \le p < \infty \, . 
\end{align}
\end{prop}
\begin{proof}
(\ref{3.1}) for the case $p=4$ was proven by \cite{T}, Prop. 4.1. Alternatively we may use
\cite{KMBT} (appendix by D. Tataru) Thm. B2:
$$ \|{\mathcal F}_t u \|_{L^2_{\tau} L^4_x} \lesssim \|u_0\|_{\dot{H}^{\frac{1}{4}}_x} \, , $$
if $u=e^{it |\nabla|}u_0$ , and ${\mathcal F}_t$ denotes the Fourier transform with respect to time. This implies by Plancherel, Minkowski's inequality and Sobolev's embedding theorem
$$ \|u\|_{L^p_x L^2_t} = \|{\mathcal F}_t u\|_{L^p_x L^2_{\tau}} \lesssim \|{\mathcal F}_t u\|_{L^2_{\tau} L^p_x} \lesssim \|{\mathcal F}_t u\|_{L^2_{\tau} H^{\frac{3}{4}-\frac{3}{p},4}_x} \lesssim \|u_0\|_{H^{1-\frac{3}{p},2}_x} \, . $$
The transfer principle (cf. e.g. \cite{KS}) gives (\ref{3.1}). (\ref{3.2}) follows similarly using $H^{\frac{3}{4}+,4}_x  \hookrightarrow L^{\infty}_x$ . 
For the proof of (\ref{3.3}) we start with the standard Strichartz-estimate $\|u\|_{L^4_{xt}} \lesssim \|u\|_{X^{\half,\half+}_{|\tau|=|\xi|}}$ and interpolate this with (\ref{3.1}) for $p=4$, which implies (\ref{3.3}) in the case $p=4$ , namely $\|u\|_{L^4_x L^{2+}_t} \lesssim \|u\|_{H^{\frac{1}{4}+,\half+}}$ . Interpolation of (\ref{3.2}) with the Sobolev type estimate $\|u\|_{L^{\infty}_{xt}} \lesssim \|u\|_{X^{\frac{3}{2}+,\half+}_{|\tau|=|\xi|}} $ implies (\ref{3.3}) in the case $p=\infty$ . Interpolating (\ref{3.3}) for $p=4$ and $p=\infty$ implies the general case $4 \le p \le \infty$ .
\end{proof}

\section{Proof of Proposition \ref{Prop.1}}
\begin{proof} [Proof of (\ref{4.1})]
We obtain with the Riesz transform $R^k= (i|\nabla|)^{-1} \partial_k$ :
$$ |\nabla|	^{-1} \left(\nabla \times \left(\begin{array}{c} \langle \psi,\alpha_1\psi \rangle \\ \langle \psi,\alpha_2 \psi \rangle\\ \langle \psi,\alpha_3\psi \rangle \end{array} \right)\right)_i = \epsilon_{ikj} R^k \langle \psi_1,\alpha^j \psi_2 \rangle \, , $$
where $\epsilon_{ikj}$ is the totally skewsymmetric tensor mith $\epsilon_{123}=1$ . We recall the definition $P A= |\nabla|^{-2} \nabla \times (\nabla \times A)$ , so that it suffices to prove
\begin{align*}
 \left\|P \left(\begin{array}{c} \langle \psi,\alpha_1\psi \rangle \\ \langle \psi,\alpha_2 \psi \rangle\\ \langle \psi,\alpha_3\psi \rangle \end{array} \right)\right\|_{X^{r-1,b-1+}_{\pm}} &\lesssim \sum_i \|\epsilon_{ikj} \sum_{\pm_1,\pm_2} R^k_{\pm} \langle \psi_{1,\pm_1} , \alpha^j \psi_{2,\pm_2} \rangle\|_{X^{r-1,b-1+}_{\pm}} \\
&\lesssim \|\psi_{1,\pm_1}\|_{X^{s,\half+}_{\pm_1}} \|\psi_{2,\pm_2}\|_{X^{s,\half+}_{\pm_2}} \, . 
\end{align*}
We now use the identity (\ref{2.7}), which implies
\begin{align*}
\epsilon_{ikj} R^k_{\pm} \langle \psi_{1,\pm_1} ,\alpha^j\psi_{2,\pm_2} \rangle & =\epsilon_{ikj} R^k_{\pm} \langle \psi_{1,\pm_1} ,\Pi_{\mp_2}(\alpha^j \psi_{2,\pm_2}) \rangle\\
& \quad -\epsilon_{ikj} R^k_{\pm} \langle \psi_{1,\pm_1} ,R^j{\pm_2 } \psi_{2,\pm_2} \rangle \, = \, I + II
\end{align*}
Using that $\Pi_{\pm_1}$ is a projector we obtain
$$ I = \epsilon_{ikj} R^k_{\pm} \langle \psi_{1,\pm_1} ,\Pi_{\pm_1}\Pi_{\mp_2}(\alpha^j \psi_{2,\pm_2}) \rangle \, .$$  The crucial point now is that by \cite{AFS}, Lemma 2  
\begin{equation}
\label{AFS}
|\Pi(\xi_1) \Pi(-\xi_2) z| \lesssim |z| \angle(\xi_1,\xi_2) \, , 
\end{equation}
so that $I$ is a null form. This implies that we may ignore the factor $\epsilon_{ikj} R^k_{\pm}$ and it suffices to prove
\begin{equation}
\label{1.a}
\|B_{\pm_1,\pm_2}(\psi_{1,\pm_1},\psi_{2,\pm_2})\|_{X^{r-1,b-1+}_{\pm}} \lesssim \|\psi_{1,\pm_1}\|_{X^{s,\half+}_{\pm_1}}\|\psi_{2,\pm_2}\|_{X^{s,\half+}_{\pm_2}} \, . 
\end{equation}
Concerning $II$ we have to prove by duality
\begin{align*}
& \left| \int \epsilon_{ikj} \langle \psi_{1,\pm_1} , R^j_{\pm_2}\psi_{2,\pm_2} \rangle \overline{R^k_{\pm} \phi_{\pm}} \,dx\, dt \right| \\
& \lesssim \|\psi_{1,\pm_1}\|_{X^{s,\half+}_{\pm_1}}\|\psi_{2,\pm_2}\|_{X^{s,\half+}_{\pm_2}} \|\phi_{\pm}\|_{X^{1-r,1-b-1-}_{\pm}} \, .
\end{align*}
We observe a null form of $Q^{jk}$ -type on the left hand side between the factors $\psi_{2,\pm_2}$ and $\phi_{\pm}$ . It is well-known (cf. \cite{ST} or \cite{HO}, Lemma 2.6) that the bilinear form $Q^{\gamma \beta}_{\pm_{1}, \pm_{2}}$ ,  defined by
\begin{align*}
	& Q^{\gamma \beta}_{\pm_{2}, \pm_{0}} (\phi_{2_{\pm_{2}}}, \phi_{0_{ \pm_{0}}}) 
	:= R^{\gamma}_{\pm_{2}} \phi_{2_{\pm_{2}}} R^{\beta}_{\pm_{0}} \phi_{0_{ \pm_{0}}} - R^{\beta}_{\pm_{2}} \phi_{2_{ \pm_{2}}} R^{\gamma}_{\pm_{0}} \phi_{0_{ \pm_{0}}} \, ,
\end{align*}
similarly to the standard null form $Q_{\gamma \beta}$ , which is defined by replacing the modified Riesz transforms $R^{\mu}_{\pm}$ by $\partial^{\mu}$, fulfills the following estimate:
$$  Q^{\gamma \beta}_{\pm_{2}, \pm_{0}} (\phi_{2_{ \pm_{2}}}, \phi_{0_{ \pm_{0}})} \precsim B_{\pm_2,\pm_0}(\psi_{2_{\pm_2}},\psi_{0_{\pm_0}} )\, , $$
where $B_{\pm_2,\pm_0}(\psi_{2_{\pm_2}},\psi_{0_{\pm_0}}) $ is defined in (\ref{2}).

Thus we have to prove :
\begin{equation}
\label{1.b}
\|B_{\pm_1,\pm_2}(\psi_{1,\pm_1},\phi_{\pm})\|_{X^{-s,-\half-}_{\pm_1}} \lesssim \|\psi_{2,\pm_2}\|_{X^{s,\half+}_{\pm_2}}\|\phi_{\pm}\|_{X^{1-r,1-b-}_{\pm}} \, . 
\end{equation}
The estimates (\ref{1.a}) and (\ref{1.b}) are proven by Prop. \ref{Prop.1.2} and Lemma \ref{Lemma}, respectively. We choose $b= \frac{7}{8}+$ .\\
 For (\ref{1.a}) we use the parameters in  Prop. \ref{Prop.1.2} as follows:
$\sigma_0=1-r$ , $\beta_0=1-b-=\frac{1}{8}-$, $\sigma_1=\sigma_2=s$ and $\beta_1=\beta_2= \half+ $ . We require $s \ge r-1$ and $\sigma_0+\sigma_1+\sigma_2 + \beta_0 = 1-r+\frac{1}{8}+2s- > 1 $ $\Leftrightarrow$ $2s-r >-\frac{1}{8}$ , one or our asumptions. Moreover we need $4s-r > \frac{1}{4}$ , which immediately follows for $s >\frac{1}{4}$ , and $3s-2r > -1$ , which we also assumed. \\
For (\ref{1.b}) we choose in Lemma \ref{Lemma}: $s_0=s$ ,  $s_1 = s$ ,  , $s_2=1-r$ , $b_2 = \frac{1}{8}-$ . We require
$s_0+s_1+s_2+b_2 = 2s+1-r+\frac{1}{8} > 1 $ $\Leftrightarrow$ $2s-r > - \frac{1}{8}$ , one of our asumptions. Moreover we need
$(s_0+s_1+s_2+b_2)+(s_0+s_1) = (2s+1-r+\frac{1}{8}) + 2s> \frac{3}{2}$ $\Leftrightarrow$ $4s-r > \frac{3}{8}$ , which follows from our assumptions $2s-r > - \frac{1}{8}$ and $s > \frac{1}{4}$  and finally $(s_0+s_1+s_2+\half)+(s_0+s_2) = 2s+1-r+\half+s+1-r > \frac{3}{2}$ $\Leftrightarrow$ $3s-2r > -1$ , one of our assumptions.
\end{proof}

\begin{proof} [Proof of (\ref{4.4})] We start with (\ref{2.7}) which implies
$$ \Pi_{\pm}(A^{df}_{j,\pm_1} \alpha^j \psi_{\pm_2}) =  \Pi_{\pm}(A^{df}_{j,\pm_1} \Pi_{\mp_2} \alpha^j \psi_{\pm_2}) - \Pi_{\pm}(A^{df}_{j,\pm_1} R^j_{\pm_2} \psi_{\pm_2}) = I+II \, . $$
The desired estimate for $I$ reduces by duality to
\begin{align*}
&\left| \int A^{df}_{j,\pm_1} \langle \psi_{0_{\pm_0}},\Pi_{\mp_2} \alpha^j \psi_{\pm_2}  \rangle \, dt \, dx  \right| \\& \quad \quad \lesssim \|\psi_{0_{\pm_0}}\|_{X^{-s,-\half-}_{\pm_0}}  \|\psi_{\pm_2}\|_{X^{s,-\half+}_{\pm_2}} \|A^{df}_{j,\pm_1}\|_{X^{r,b}_{\pm_1}} \, .
\end{align*}
The left hand side equals
$$\left| \int A^{df}_{j,\pm_1} \langle \Pi_{\mp_2} \Pi_{\pm_0} \psi_{0_{\pm_0}},\Pi_{\mp_2} \alpha^j \psi_{\pm_2}  \rangle \, dt \, dx  \right| \, , $$
which contains a null form between $\psi_{0,\pm_0}$ and $\psi_{2,\pm_2}$ by (\ref{AFS}), so that it remains to prove
\begin{equation}
\label{4.a}
\|B_{\pm_0,\pm_2}(\psi_{0,\pm_0} , \psi_{\pm_2})\|_{X^{-r,-b}_{\pm_1}} \lesssim \|\psi_{0,\pm_0}\|_{X^{-s,\half-}_{\pm_0}} \|\psi_{\pm_2}\|_{X^{s,\half+}_{\pm_2}}  \, .
\end{equation}
In order to prove that II contains also a null form, we use the well-known identity (by \cite{KM} or \cite{ST2}):
\begin{equation}
\label{Adf}
A^{df}_j R^j_{\pm_1} \psi_{\pm_1} = -\epsilon^{jkl} \partial_k w_l R^j_{\pm_1} \psi_{\pm_1} = (\nabla w_l \times \frac{\nabla}{|\nabla|} \psi_{\pm_1})^l \, , 
\end{equation}
where 
$$w=|\nabla|^{-2} \nabla \times A = |\nabla|^{-2} \nabla \times A^{df} + |\nabla|^{-2} \nabla \times A^{cf}\, .$$

The last summand results in a term of the type $A^{cf}\psi$ , if we ignore its special structure, which is possible for our purposes. We postpone its estimate to the proof of (\ref{4.3}) below, where a similar estimate is proven.

The first summand results in a  $Q^{ij}$-type null form, which is essentially of the type $Q^{ij}(A^{df},\psi_{\pm_1})$ . This implies (cf. \cite{ST} or \cite{HO}, Lemma 2.6) that it remains to prove for this part the following estimate:
\begin{equation}
\label{4.b}
\| B_{\pm_2,\pm_1}(A^{df}_{j,\pm_2} , \psi_{\pm_1}) \|_{X^{s,-\half++}_{\pm}} \lesssim \|A^{df}_{j,\pm_2} \|_{X^{r,b}_{\pm_1}} \|\psi_{\pm_1}\|_{X^{s,\half+}_{\pm_2}} \, .
\end{equation}
We now prove (\ref{4.a}) and (\ref{4.b}) by Prop. \ref{Prop.1.2}. For (\ref{4.a}) we choose the parameters in this proposition as follows:
$\sigma_0 = r$ , $\sigma_1 = -s$ , $\sigma_2 = s$ , $\beta_0 = b = \frac{7}{8}+$ , $\beta_1 = \beta_2 = \half+$ . We require $ r \ge s$ and $\sigma_0+\sigma_1+\sigma_2 + \beta_0 = r+\frac{7}{8} > 1$ , thus $r > \frac{1}{8}$ . Moreover we need $2r+s > 1$ , which holds for $r>\frac{5}{8}$ and $s > \frac{1}{4}$ , and $2r-s > 1$ as assumed. Next, for (\ref{4.b}) we have to choose $\sigma_0=-s$ , $\sigma_1=r$ , $\sigma_2 = s$ , $\beta_0=\beta_2=\half+$ , $\beta_1 = b = \frac{7}{8}+$ . This requires  $r > \half$ , $2r+s > \half$ , which is satisfied. Moreover we need $r>\frac{3}{2}-(-s+b+s) $ $\Leftrightarrow$ $r> \frac{3}{2}- b = \frac{5}{8}-$  and $2r-s > 1$ , which holds by our assumptions.
\end{proof}

\begin{proof} [Proof of (\ref{4.3})]
We even prove the estimate with $X^{s,-\frac{1}{2}++}_{|\tau|=|\xi|}$ replaced by
$X^{s,0}_{|\tau|=|\xi|}$ on the left hand side. Morover we remark that the matrices $\alpha^j$ are completely irrelevant for the estimate. This also implies that the estimate for the term of type $A^{cf} \psi$ in the proof of (\ref{4.4}) , which we postponed, is also proven here.

We may reduce to
\begin{align*}
 \int_* \frac{\widehat{u}_1(\tau_1,\xi_1)}{\langle  \xi_1\rangle^l \langle \tau_1 \rangle^{\frac{1}{2}+}} 
\frac{\widehat{u}_2(\tau_2,\xi_2)}{\langle \xi_2 \rangle^s\langle |\tau_2| - |\xi_2|\rangle^{\frac{1}{2}+}}\langle \xi_3 \rangle^s
\widehat{u}_3(\tau_3,\xi_3) d\xi d\tau
\lesssim \prod_{i=1}^3 \|u_i\|_{L^2_{xt}} \, ,
\end{align*}
where * denotes integration over $\xi = (\xi_1,\xi_2,\xi_3) , \tau=(\tau_1,\tau_2,\tau_3)$ with $\xi_1+\xi_2+\xi_3=0$ and $\tau_1+\tau_2+\tau_3 =0$. We assume here and in the following without loss of generality that the Fourier transforms are nonnnegative. \\
Case 1: $|\xi_1| \ge |\xi_2|$ $\Rightarrow$ $\langle \xi_3 \rangle^s \lesssim \langle \xi_1 \rangle^s $ .\\
The estimate reduces to
\begin{align*}
 \int_* \frac{\widehat{u}_1(\tau_1,\xi_1)}{ \langle \tau_1 \rangle^{\frac{1}{2}+}} 
\frac{\widehat{u}_2(\tau_2,\xi_2)}{\langle \xi_2 \rangle^l \langle |\tau_2| - |\xi_2|\rangle^{\frac{1}{2}+}}
\widehat{u}_3(\tau_3,\xi_3) d\xi d\tau 
\lesssim \prod_{i=1}^3 \|u_i\|_{L^2_{xt}} \, .
\end{align*}
This follows under the assumption $l > 1$ from the estimate
\begin{align}
\nonumber
\Big|\int v_1 v_2 v_3 dx dt \Big| & \lesssim \|v_1\|_{L^2_x L^{\infty}_t} \|v_2\|_{L^{\infty}_x L^2_t} \|v_3\|_{L^2_x L^2_t} \\
\label{20}
&\lesssim \|v_1\|_{X^{0,\frac{1}{2}+}_{\tau=0}} \|v_2\|_{X^{1+,\frac{1}{2}+}_{|\tau|=|\xi|}} 
\|v_3\|_{X^{0,0}_{|\tau|=|\xi|}} \, ,
\end{align}
where we used (\ref{3.2}) for the second factor.\\
Case 2: $|\xi_2| \ge |\xi_1|$ $\Rightarrow$ $\langle \xi_3 \rangle^s \lesssim \langle \xi_2 \rangle^s $.\\
In this case the desired estimate follows from
\begin{equation}
\label{21}
\int_* m(\xi_1,\xi_2,\xi_3,\tau_1,\tau_2,\tau_3) \widehat{u}_1(\xi_1,\tau_1)  \widehat{u}_2(\xi_2,\tau_2) \widehat{u}_3(\xi_3,\tau_3) d\xi d\tau \lesssim \prod_{i=1}^3 \|u_i\|_{L^2_{xt}} \, , 
\end{equation}
where 
$$ m = \frac{1}{ \langle |\tau_2| - |\xi_2|\rangle^{\frac{1}{2}+}  \langle \xi_1 \rangle^l \langle \tau_1 \rangle^{\frac{1}{2}+}} \, .$$
We imitate Tao's proof \cite{T} for a similar estimate. \\
By two applications of the averaging principle (\cite{T1}, Prop. 5.1) we may replace $m$ by
$$ m' = \frac{ \chi_{||\tau_2|-|\xi_2||\sim 1} \chi_{|\tau_1| \sim 1}}{ \langle \xi_1 \rangle^l} \, . $$
Let now $\tau_2$ be restricted to the region $\tau_2 =T + O(1)$ for some integer $T$. Then $\tau_3$ is restricted to $\tau_3 = -T + O(1)$, because $\tau_1 + \tau_2 + \tau_3 =0$, and $\xi_2$ is restricted to $|\xi_2| = |T| + O(1)$. The $\tau_3$-regions are essentially disjoint for $T \in {\mathbb Z}$ and similarly the $\tau_2$-regions. Thus by Schur's test (\cite{T1}, Lemma 3.11) we only have to show
\begin{align*}
 &\sup_{T \in {\mathbb Z}} \int_* \frac{ \chi_{\tau_3=-T+O(1)} \chi_{\tau_2=T+O(1)} \chi_{|\tau_1|\sim 1} \chi_{|\xi_2|=|T|+O(1)}}{\langle \xi_1 \rangle^l}\cdot \\
 & \hspace{14em} \cdot\widehat{u}_1(\xi_1,\tau_1) \widehat{u}_2(\xi_2,\tau_2)
\widehat{u}_3(\xi_3,\tau_3) d\xi d\tau \lesssim \prod_{i=1}^3 \|u_i\|_{L^2_{xt}} \, . 
\end{align*}
The $\tau$-behaviour of the integral is now trivial, thus we reduce to
\begin{equation}
\label{50}
\sup_{T \in {\mathbb N}} \int_{\sum_{i=1}^3 \xi_i =0}  \frac{ \chi_{|\xi_2|=T+O(1)}}{ \langle \xi_1 \rangle^l} \widehat{f}_1(\xi_1)\widehat{f}_2(\xi_2)\widehat{f}_3(\xi_3)d\xi \lesssim \prod_{i=1}^3 \|f_i\|_{L^2_x} \, .
\end{equation}
We apply Schwarz' inequality so that
\begin{align*}
L.H.S. \, of \,  (\ref{50})
\lesssim \sup_{T \in{\mathbb N}} \| \chi_{|\xi|=T+O(1)} \ast \langle \xi \rangle^{-2l}\|^{\frac{1}{2}}_{L^{\infty}(\mathbb{R}^3)} \prod_{i=1}^3 \|f_i\|_{L^2_x} \lesssim \prod_{i=1}^3 \|f_i\|_{L^2_x}\, .
\end{align*}
The last estimate follows by an elementary calculation for $l>1$ ,  which completes the proof.
\end{proof}

\begin{proof} [Proof of (\ref{4.2})] We remark that for our purpose it is admissible to replace the singular operator $|\nabla|^{-1}$ by $\langle \nabla \rangle^{-1}$ in three space dimensions, where we use \cite{T1}, Cor. 8.2.

We may reduce to
\begin{align*}
 \int_* \frac{\widehat{u}_1(\tau_1,\xi_1)}{\langle  \xi_1\rangle^s \langle |\tau_1|-|\xi_1| \rangle^{\frac{1}{2}+}} 
\frac{\widehat{u}_2(\tau_2,\xi_2)}{\langle \xi_2 \rangle^s \langle |\tau_2| - |\xi_2|\rangle^{\frac{1}{2}+}} \frac{\langle \xi_3 \rangle^{l-1}
\widehat{u}_3(\tau_3,\xi_3)}{\langle \tau_3 \rangle^{\frac{1}{2}-}} d\xi d\tau
\lesssim \prod_{i=1}^3 \|u_i\|_{L^2_{xt}} \, .
\end{align*}
We assume without loss of generality $|\xi_1| \le |\xi_2|$ , so that $|\xi_3| \lesssim |\xi_2|$ . \\
It suffices to prove
\begin{align*}
 \int_* \frac{\widehat{u}_1(\tau_1,\xi_1)}{\langle  \xi_1\rangle^s \langle |\tau_1|-|\xi_1| \rangle^{\frac{1}{2}+}} 
\frac{\widehat{u}_2(\tau_2,\xi_2)}{ \langle  \xi_2\rangle^{s+1-l}\langle |\tau_2|-|\xi_2| \rangle^{\frac{1}{2}+}} \frac{
\widehat{u}_3(\tau_3,\xi_3)}{\langle \tau_3 \rangle^{\frac{1}{2}-}} d\xi d\tau 
\lesssim \prod_{i=1}^3 \|u_i\|_{L^2_{xt}} \, ,
\end{align*}
which is implied under our assumption $s-l \ge -\frac{3}{4}$ by
\begin{align*}
\Big|\int v_1 v_2 v_3 dx dt \Big| & \lesssim \|v_1\|_{L^4_x L^{2+}_t} \|v_2\|_{L^4_x L^2_t} \|v_3\|_{L^2_x L^{\infty -}_t} \\
&\lesssim \|v_1\|_{X^{\frac{1}{4}+,\frac{1}{2}+}_{|\tau|=|\xi|}} \|v_2\|_{X^{\frac{1}{4},\half+}_{|\tau|=|\xi|}} 
\|v_3\|_{X^{0,\frac{1}{2}-}_{\tau =0}} \, ,
\end{align*}
where we used (\ref{3.3}) in the case $p=4$ for the first two factors and Sobolev's embedding theorem for the last one.
\end{proof}

\section{The result in two space dimensions}

If we apply the projection $P$ of $A$ onto the divergence-free part $A^{df}$ we obtain from (\ref{1'}) the system
\begin{equation}
\square A^{df} = - P \left(\begin{array}{c} \langle \psi,\alpha_1\psi \rangle \\ \langle \psi,\alpha_2 \psi \rangle \end{array} \right) \, .
\end{equation}
By the definition of $A^{cf}$ the equation (\ref{2'}) can be rewritten as
\begin{equation}
\partial_t A^{cf} = -|\nabla|^{-2} \nabla \langle \psi, \psi \rangle \, . \end{equation} 

As in the three-dimensional case  the proof of the local existence theorem reduces to the following bilinear estimates.

\begin{prop}
\label{Prop.1'}
Let the assumptions on  $s$,$r$ and $l$ in Theorem \ref{Theorem1.1} be fulfilled. Let $ b = \half+$ . Then the estimates (\ref{4.1}),(\ref{4.3}),(\ref{4.4}) apply and
\begin{align}
\label{4.2'}
\| |\nabla|^{-1+\epsilon} \langle \psi_{\pm_1} , \psi_{\pm_2} \rangle \|_{X^{l-\epsilon,-\half++}_{\tau=0}} & \lesssim \|\psi_{\pm_1}\|_{X^{s,\half+}_{\pm_1}} \|\psi_{\pm_2}\|_{X^{s,\half+}_{\pm_2}}  
\end{align}
for a sufficiently small $\epsilon > 0$ ,
where $\pm$ , $\pm_1$ and $\pm_2$ denote independent signs.
\end{prop}

The following  bilinear estimates in wave-Sobolev spaces were proven in the two-dimensional case in \cite{AFS0}.
\begin{prop}
\label{Prop.2}
For $s_0,s_1,s_2,b_0,b_1,b_2 \in {\mathbb R}$ and $u,v \in   {\mathcal S} ({\mathbb R}^{2+1})$ the estimate
$$\|uv\|_{H^{-s_0,-b_0}} \lesssim \|u\|_{H^{s_1,b_1}} \|v\|_{H^{s_2,b_2}} $$ 
holds, provided the following conditions are satisfied:
\begin{align*}
\nonumber
& b_0 + b_1 + b_2 > \frac{1}{2} \, ,
& b_0 + b_1 \ge 0 \, ,\quad \qquad  
& b_0 + b_2 \ge 0 \, ,
& b_1 + b_2 \ge 0
\end{align*}
\begin{align*}
\nonumber
&s_0+s_1+s_2 > \frac{3}{2} -(b_0+b_1+b_2) \\
\nonumber
&s_0+s_1+s_2 > 1 -\min(b_0+b_1,b_0+b_2,b_1+b_2) \\
\nonumber
&s_0+s_1+s_2 > \frac{1}{2} - \min(b_0,b_1,b_2) \\
\nonumber
&s_0+s_1+s_2 > \frac{3}{4} \\
 &(s_0 + b_0) +2s_1 + 2s_2 > 1 \\
\nonumber
&2s_0+(s_1+b_1)+2s_2 > 1 \\
\nonumber
&2s_0+2s_1+(s_2+b_2) > 1 
\end{align*}
\begin{align*}
\nonumber
&s_1 + s_2 \ge \max(0,-b_0) \, ,\quad
\nonumber
s_0 + s_2 \ge \max(0,-b_1) \, ,\quad
\nonumber
s_0 + s_1 \ge \max(0,-b_2)   \, .
\end{align*}
\end{prop}

\begin{lemma} 
\label{Lemma1}
Let  $s_0+s_1\ge 0$ , $s_0+s_2 \ge 0 $ , $s_1+s_2+\half \ge 0$  and $u,v \in   {\mathcal S} ({\mathbb R}^{2+1})$ . Assume that
$$s_0+s_1+s_2  > \frac{1}{4} \, , $$
$$ (s_0+s_1+s_2+\half)+(s_0+s_1) >1 \, , $$
$$(s_0+s_1+s_2+\half) +(s_0+s_2) > 1$$
Then we have the following estimate
$$
  \norm{B_{(\pm_1 \xi_1,\pm_2 \xi_2)}(u,v)}_{H^{-s_0,-\half+}}
  \lesssim
  \norm{u}_{X^{s_1,\half-}_{\pm_1}} \norm{v}_{X^{s_2,\half-}_{\pm 2}}\, .
$$
\end{lemma}
\begin{proof}
The proof is similar to the proof of Lemma \ref{Lemma} by use of Prop. \ref{Prop.2} and the estimate (\ref{angle}) and omitted.
\end{proof}

The following proposition replaces Prop. \ref{Prop.2.4}.

\begin{lemma}
\label{Lemma2}
For $2 \le p \le 6$ the following estimates hold:
\begin{align*}
\|u\|_{L^p_x L^2_t} & \lesssim \|u\|_{X^{\frac{1}{2}(\frac{1}{2}-\frac{1}{p}),\frac{3}{2}(\frac{1}{2}-\frac{1}{p})+}_{|\tau|=|\xi|}} \, , \\
\|u\|_{L^p_x L^{2+}_t} & \lesssim \|u\|_{X^{\frac{1}{2}(\frac{1}{2}-\frac{1}{p})+,\frac{3}{2}(\frac{1}{2}-\frac{1}{p})+}_{|\tau|=|\xi|}} \, .
\end{align*}
\end{lemma}
\begin{proof}
By \cite{KMBT}, Thm. B2 we obtain
$ \|{\mathcal F}_t u \|_{L^2_{\tau} L^6_x} \lesssim \|u_0\|_{\dot{H}^{\frac{1}{6}}} \, , $
if $u= e^{it|\nabla|} u_0$ and ${\mathcal F}_t$ denotes the Fourier transform with respect to time. This implies by Plancherel and Minkowski's inequality
$$ \|u\|_{L^6_x L^2_t} = \|{\mathcal F}_t u \|_{L^6_x L^2_{\tau}} \le \|{\mathcal F}_t u \|_{L^2_{\tau} L^6_x} \lesssim \|u_0\|_{\dot{H}^{\frac{1}{6}}} \, . $$
The transfer principle (cf. e.g. \cite{KS}) implies 
\begin{equation}
\|u\|_{L^6_x L^2_t} \lesssim \|u\|_{X^{\frac{1}{6},\frac{1}{2}+}_{|\tau|=|\xi|}} \, .
\end{equation}
Interpolation with the standard Strichartz estimate 
(combined with the transfer principle)
$
\|u\|_{L^6_{xt}} \lesssim \|u\|_{X^{\frac{1}{2},\frac{1}{2}+}_{|\tau|=|\xi|}} 
$
gives
\begin{equation}
\|u\|_{L^6_x L^{2+}_t} \lesssim \|u\|_{X^{\frac{1}{6}+,\frac{1}{2}+}_{|\tau|=|\xi|}} \, .
\end{equation}
Interpolation of the last two inequalities with the trivial identity
$\|u\|_{L^2_x L^2_t} = \|u\|_{X^{0,0}_{|\tau|=|\xi|}} $ completes the proof.
\end{proof}

Next we prove Proposition \ref{Prop.1'} in the two-dimensional case, i.e.   the bilinear estimates (\ref{4.1}),(\ref{4.2'}),(\ref{4.3}) and (\ref{4.4}).
\begin{proof}[Proof of (\ref{4.1})]
For $A= \left(\begin{array}{c} \langle \psi,\alpha_1\psi \rangle \\ \langle \psi,\alpha_2 \psi \rangle \end{array} \right)$ we obtain
$R_1 A_2-R_2 A_1 = \epsilon_{kj} R^k \langle \psi,\alpha^j \psi \rangle \, , $
Recalling the definition  $PA:=  (R_2(R_1 A_2-R_2 A_1),-R_1(R_1 A_2 - R_2 A_1))$  , it suffices to prove
\begin{align*}
 \left\|P \left(\begin{array}{c} \langle \psi,\alpha_1\psi \rangle \\ \langle \psi,\alpha_2 \psi \rangle \end{array} \right)\right\|_{X^{r-1,-\half++}_{\pm}} &\lesssim  \|\epsilon_{kj} \sum_{\pm_1,\pm_2} R^k_{\pm} \langle \psi_{1,\pm_1} , \alpha^j \psi_{2,\pm_2} \rangle\|_{X^{r-1,-\half++}_{\pm}} \\
&\lesssim \|\psi_{1,\pm_1}\|_{X^{s,\half+}_{\pm_1}} \|\psi_{2,\pm_2}\|_{X^{s,\half+}_{\pm_2}} \, . 
\end{align*}
Exactly as in the three-dimensional case we have to prove (\ref{1.a}) and (\ref{1.b}). By Lemma \ref{Lemma1} in the case $b=\half+$ the following conditions are required as one easliy checks: for (\ref{1.a}): $2s-r > -\frac{3}{4}$ and $3s-2r >-\frac{3}{2}$ ,
and for (\ref{1.b}): $3s-2r > - \frac{3}{2}$ and $4s-r > -\half$.
These conditions are assumed.
\end{proof}

\begin{proof}[Proof of (\ref{4.4})]
Following the proof in three dimensions we replace (\ref{Adf}) by
\begin{align*}
A^{df}_j R^j_{\pm_1} \psi_{\pm_1} & = R_2 |\nabla|^{-1} (\nabla \times A) R^1_{\pm_1} \psi_{\pm_1} - R_1 |\nabla|^{-1} (\nabla \times A) R^2_{\pm_1} \psi_{\pm_1} \\
& = \epsilon_{ij} R^j |\nabla|^{-1}  (\nabla \times A) R^i_{\pm_1} \psi_{\pm_1} \\
&=\epsilon_{ij} R^j |\nabla|^{-1}  (\nabla \times A^{df}) R^i_{\pm_1} \psi_{\pm_1}
+\epsilon_{ij} R^j |\nabla|^{-1}  (\nabla \times A^{cf}) R^i_{\pm_1} \psi_{\pm_1} \\
& \precsim Q_{21}(A^{df},\psi_{\pm_1}) + A^{cf} \psi_{\pm_1} \, ,
\end{align*}
where $\nabla \times A := R_1 A_2-R_2 A_1$ , and $u \precsim v$ means $|\widehat{u}| \lesssim |\widehat{v}|$ . As in the three-dimensional case we reduce to the estimates (\ref{4.a}) and (\ref{4.b}), which by Lemma \ref{Lemma1} for $b=\half+$ require the following conditions for (\ref{4.a}): $2r-s > \half$ and $r > \frac{1}{4}$ , and for (\ref{4.b}): $2r-s > \half$ , which are satisfied.
\end{proof}

\begin{proof}[Proof of (\ref{4.3})]
We use the assumption $l > \half$ and obtain the following analogue of (\ref{20}):
\begin{align*}
\nonumber
\Big|\int v_1 v_2 v_3 dx dt \Big| & \lesssim \|v_1\|_{L^2_x L^{\infty}_t} \|v_2\|_{L^{\infty}_x L^2_t} \|v_3\|_{L^2_x L^2_t} \\
&\lesssim \|v_1\|_{X^{0,\frac{1}{2}+}_{\tau=0}} \|v_2\|_{X^{\half+,\frac{1}{2}+}_{|\tau|=|\xi|}} 
\|v_3\|_{X^{0,0}_{|\tau|=|\xi|}} \, ,
\end{align*}
where we used (\ref{3.2}) for the second factor. The rest of the proof is easily reduced to the estimate
\begin{align*}
\sup_{T \in{\mathbb N}} \| \chi_{|\xi|=T+O(1)} \ast \langle \xi \rangle^{-2l}\|_{L^{\infty}(\mathbb{R}^2)}  \lesssim 1 \, ,
\end{align*}
which follows by an elementary calculation for $l>\half$ .
\end{proof}

\begin{proof} [Proof of (\ref{4.2'})] We replace the singular operator $|\nabla|^{-1+\epsilon}$ by $\langle \nabla \rangle^{-1+\epsilon}$ for $\epsilon > 0$ in two space dimensions by \cite{T1}, Cor. 8.2. Arguing as in the three-dimensional case and assuming $s > 0$ and $s-l \ge -\half$  we reduce to
\begin{align*}
\Big|\int v_1 v_2 v_3 dx dt \Big| & \lesssim \|v_1\|_{L^{2+}_x L^{2+}_t} \|v_2\|_{L^{\infty-}_x L^2_t} \|v_3\|_{L^2_x L^{\infty -}_t} \\
&\lesssim \|v_1\|_{X^{0+,0+}_{|\tau|=|\xi|}} \|v_2\|_{X^{\frac{1}{2},\half+}_{|\tau|=|\xi|}} 
\|v_3\|_{X^{0,\frac{1}{2}-}_{\tau =0}} \, ,
\end{align*}
where we used (\ref{3.3}) for the first two factors and Sobolev's embedding theorem for the last one.
\end{proof}

\end{document}